\documentclass[12pt, reqno]{amsart}
\usepackage{amssymb}
\usepackage{amsrefs}
\usepackage{amsmath}
\usepackage[usenames]{color}
 \usepackage{fourier} 
 \newtheorem{thm}{}[section]
 \newtheorem{theorem}[thm]{Theorem}
 \newtheorem{corollary}[thm]{Corollary}

 \theoremstyle{definition}
 
 \theoremstyle{remark}

 \numberwithin{equation}{section}
 \allowdisplaybreaks
 
\newcommand{\NN}{\ensuremath{\mathbb{N}}}

\newcommand{\RR}{\ensuremath{\mathbb{R}}}

\newcommand{\XX}{\ensuremath{\mathbb{X}}}

\newcommand{\xx}{\ensuremath{\mathbf{e}}}

\newcommand{\BB}{\mathcal{B}}

\newcommand{\GG}{\mathcal{G}}

\begin{document}

\title{Characterization of $1$-quasi-greedy bases}
\author[F. Albiac]{F. Albiac}\address{Mathematics Department\\ Universidad P\'ublica de Navarra\\Campus de Arrosad\'{i}a\\
 Pamplona\\ 31006 Spain\\
 Tel.: +34-948-169553\\
 Fax: +34-948-166057\\
} \email{fernando.albiac@unavarra.es}
 \author[J. L. Ansorena]{J. L. Ansorena}\address{Department of Mathematics and Computer Sciences\\
Universidad de La Rioja\\ Edificio Luis Vives\\
 Logro\~no\\
26004 Spain\\
 Tel.: +34-941299464 \\
 Fax: +34-941299460} \email{joseluis.ansorena@unirioja.es}



\begin{abstract} We show that a (semi-normalized) 
basis 
in a Banach space 
 is  quasi-greedy with quasi-greedy constant equal to $1$ if and only if 
 it is unconditional with suppression-unconditional constant equal to $1$.  \end{abstract}


\maketitle

\section{Introduction and background}\label{intro}
\noindent Let $(\XX, \Vert\cdot\Vert)$ be an infinite-dimensional Banach space, and let $\BB=(\xx_n)_{n=1}^\infty$ be a semi-normalized  basis for $\XX$ with biorthogonal functionals 
$(\xx_n^*)_{n=1}^\infty$.
The basis $\BB$ is  \textit {quasi-greedy} if for any $x\in \XX$ the corresponding series expansion, 
\[
x=\sum_{n=1}^\infty \xx_n^*(x)\xx_n
\]
converges in norm after  reordering it so that the   sequence  $(|\xx_n^*(x)|)_{n=1}^\infty$ is decreasing. Wojtaszczyk showed  \cite{Wo2000} that a basis $(\xx_n)_{n=1}^\infty$  of $\XX$ is quasi-greedy if and only if the greedy operators  $\GG_N\colon \XX\to \XX$ defined by
\[
x=\sum_{j=1}^{\infty}\xx_j^{\ast}(x)\xx_{j} \mapsto \GG_{N}(x)=\sum_{j\in \Lambda_{N}(x)}\xx_j^{\ast}(x)\xx_{j},
\]
where $\Lambda_{N}(x)$ is any $N$-element set of indices such that 
\[
\min\{|\xx_{j}^{\ast}(x)| \colon  j\in \Lambda_{N}(x)\}\ge \max\{|\xx_{j}^{\ast}(x)| \colon j\not\in \Lambda_{N}(x)\},
\]
are uniformly bounded, i.e.,
\begin{equation}\label{WoCondition}
\Vert \GG_N(x)\Vert\le C \Vert x\Vert, \quad x \in \XX, \,  N\in\NN.
\end{equation}
for some constant $C$ independent of $x$ and $N$. Note that the operators $(\GG_{N})_{N=1}^{\infty}$ are neither linear nor continuous, so this is not just the Uniform Boudedness Principle!

 Obviously, If \eqref{WoCondition} holds  then there is a (possibly different) constant $\tilde{C}$ such that 
\begin{equation}\label{SupWoCondition}
\Vert x-\GG_N(x)\Vert\le \tilde{C} \Vert x\Vert, \quad x \in \XX, \,  N\in\NN.
\end{equation}
We will denote by $C_w$ the smallest constant such that \eqref{WoCondition}  holds,  and by  $C_t$ the least constant in \eqref{SupWoCondition}. It is rather common (cf.\ \cites{GHO2013,DKKT2003}) and convenient to define the {\it quasi-greedy constant} $C_{qg}$ of the basis as
\begin{equation*}\label{WojtCharacterization}
C_{qg}=\max\{ C_w, C_t \}.
\end{equation*}
If  $\BB$ is a quasi-greedy basis and $C$ is a  constant such that $C_{qg}\le C$ we will say that $\BB$ is  {\it $C$-quasi-greedy}.

 Recall also that a basis  $(\xx_n)_{n=1}^\infty$ in a Banach space $\XX$ is {\it unconditional}   if for any $x\in \XX$ the series $\sum_{n=1}^\infty \xx_n^*(x)\xx_n$ converges in norm to $x$ regardless of the order in which we arrange
 the terms. The property of being unconditional is easily seen to be equivalent to that of being {\it suppression unconditional}, which means that the natural projections onto any subsequence  of the basis
 \[
 P_A(x)=\sum_{n\in A} \xx_n^*(x) \xx_n,\quad A\subset \mathbb N,
 \]
are uniformly bounded, i.e., there is a constant $K$ such that for all $x=\sum_{n=1}^{\infty} \xx_n^*(x) \xx_n$ and all $A\subset \NN$,
\begin{equation}\label{unifbound}
\left\Vert\sum_{n\in A} \xx_n^*(x) \xx_n \right\Vert \le K \left\Vert\sum_{n=1}^{\infty} \xx_n^*(x) \xx_n \right\Vert.
\end{equation}
The smallest $K$ in \eqref{unifbound} is the {\it suppression unconditional constant} of the basis, and will be denoted by $K_{su}$. Notice that
 \[
 K_{su}=\sup  \{ \Vert P_A \Vert \colon A\subseteq \NN \text{ is finite} \}=\sup  \{ \Vert P_A \Vert \colon A\subseteq \NN \text{ is cofinite}\}.
 \]
 If  a basis $\BB$ is unconditional and $K$ is a constant such that  $ K_{su}\le K$  we will say that $\BB$  is {\it $K$-suppression unconditional}.    
 
 Quasi-greedy bases are not in general unconditional; in fact, most classical spaces contain conditional quasi-greedy bases. Wojtaszczyk gave in \cite{Wo2000} a general construction  (improved in \cite{GW2014}) to produce quasi-greedy bases in some Banach spaces. His method yields the existence of  conditional quasi-greedy bases in separable Hilbert spaces, in the spaces $\ell_{p}$ and $L_{p}[0,1]$  for $1<p<\infty$, and  in the Hardy space $H_1$. 
 Dilworth and Mitra showed in \cite{DilworthMitra2001} that $\ell_1$ also has a conditional quasi-greedy basis. In spite of  that, quasi-greedy bases preserve some vestiges of unconditionality and, for instance, they are {\it unconditional for constant coefficients} (see \cite{Wo2000}).

Conversely,  unconditional bases are always quasi-greedy. To be precise, if $\BB$ is   $K$-suppression unconditional  then $\BB$ is 
 $K$-quasi-greedy. In particular, unconditional bases with $K_{su}=1$ are quasi-greedy with $C_{w}=1$. Our aim  is to show  the converse  of this statement, thus characterizing $1$-quasi-greedy bases. The related problem of characterizing bases that are $1$-greedy was solved in \cite{AlbiacWojt2006}. This question is relevant since the optimality in the constants of  greedy-like bases seems to improve the properties of the corresponding basis. Indeed, in the 
 ''isometric case" greedy bases gain in symmetry (they are invariant under greedy permutations instead of merely democratic). Our result reinforces  this pattern  by showing that ''isometric'' quasi-greedy basis are not merely unconditional for constant coefficients but unconditional.

\section{The Main Theorem and its Proof}
 \noindent As a by-product of their research on  unconditionality-type properties of quasi-greedy bases, Garrig\'os and Wojtaszczyk  \cite{GW2014}  have shown that 
  bases in Hilbert spaces with $C_w=1$ are orthogonal. A direct proof of their result can be obtained as follows. 

Let
$\BB=(\xx_n)_{n=1}^\infty$  be a basis in a (real or complex) Hilbert space with  $C_w=1$. Then, if $|\varepsilon|=1$,  $0< t < 1$, and $i\not=j$,
\[
\Vert \xx_i \Vert^2  \le \Vert  \xx_i +\varepsilon t \xx_j\Vert^2=\Vert \xx_i\Vert^2+2 t  \Re(\varepsilon \langle \xx_i,\xx_j \rangle) + t^2 \Vert \xx_j\Vert^2.
\]
Simplifying,
\[
-2 \Re(\varepsilon \langle \xx_i,\xx_j \rangle) \le t \Vert \xx_j\Vert^2.
\]
Choosing $\varepsilon$ such that  $ \varepsilon\langle \xx_i,\xx_j \rangle =-| \langle \xx_i,\xx_j \rangle|$ and letting $t$ tend to zero we obtain $ | \langle \xx_i,\xx_j \rangle|=0$.

A strengthening of this argument leads to the following generalization of  Garrig\'os-Wojtaszczyk's result.

\begin{theorem}\label{1QGTheorem} A quasi-greedy basis $(\xx_{n})_{n=1}^{\infty}$ in a Banach space $\XX$ is quasi-greedy  with $C_w=1$ if and only if it is unconditional with suppression unconditional constant $K_{su}=1$.
\end{theorem}

\begin{proof} 
We need only show that if  $x$ and  $y$  are vectors finitely  supported in $(\xx_{n})_{n=1}^{\infty}$ with disjoint supports then $\Vert x \Vert \le \Vert x+y\Vert$. This readily implies that $(\xx_{n})_{n=1}^{\infty}$ is unconditional with suppression unconditional constant $K_{su}=1$. 

Suppose that this in not the case and that we can pick  $x$, $y\in \XX$ finitely and disjointly supported in $(\xx_{n})_{n=1}^{\infty}$  with  $\Vert x+y\Vert  <\Vert x \Vert$.
Consider the function $\varphi\colon\RR\to[0,\infty)$ defined by  
\[
\varphi(t)=\Vert x + t y\Vert.
\]
Using the definition, it is straightforward to check that $\varphi$ is a convex function on the entire real line.   Moreover,  $\varphi(0)=\Vert x\Vert$ and, by  assumption, $\varphi(1)<\Vert x\Vert$. Therefore, $\varphi(t)<\Vert x \Vert $
for all $0<t<1$. Choosing $t\in (0,1)$ small enough we have  $x=\GG_N(x+ty)$, where $N$ is the cardinal of the support of $x$. Consequently, for such a $t$,
\[
\Vert x+t y \Vert =\varphi(t)
<\Vert x \Vert =\Vert \GG_N(x+ty) \Vert \le \Vert x+t y \Vert,
\]
where we used the hypothesis on the quasi-greedy constant of the basis to obtain the last inequality. This absurdity proves the result.
\end{proof}

We close with some consequences  of Theorem~\ref{1QGTheorem}, which need no further explanation.
\begin{corollary} Suppose $\BB=(\xx_{n})_{n=1}^{\infty}$ is a basis in a Banach space $\XX$ with
 $C_w=1$. Then $C_t=1$; in particular $\BB$ is  $1$-quasi-greedy.
\end{corollary}

\begin{corollary}  If a  basis $(\xx_{n})_{n=1}^{\infty}$
 in a Banach space $\XX$ is  $1$-quasi-greedy then it is  $1$-suppression unconditional.\end{corollary}

\begin{corollary} Suppose  $\BB=(\xx_{n})_{n=1}^{\infty}$ is a basis in a Banach space $(\XX, \Vert\cdot\Vert)$.  Then $\XX$ admits an equivalent norm $\VERT\cdot\VERT$ so that $\BB$ is $1$-quasi-greedy in  the space $(\XX, \VERT\cdot\VERT)$ if and only if $\BB$ is unconditional.

\end{corollary}

\subsection*{Acknowledgement}  F. Albiac  acknowledges the support of the Spanish Ministry for Economy and Competitivity Grant {\it Operators, lattices, and structure of Banach spaces}, 
reference number MTM2012-31286.


\begin{bibsection}
\begin{biblist}

\bib{AlbiacWojt2006}{article}{
   author={Albiac, F.},
   author={Wojtaszczyk, P.},
   title={Characterization of 1-greedy bases},
   journal={J. Approx. Theory},
   volume={138},
   date={2006},
   number={1},
   pages={65--86},
   issn={0021-9045},
}

\bib{DKKT2003}{article}{
   author={Dilworth, S. J.},
   author={Kalton, N. J.},
   author={Kutzarova, D.},
   author={Temlyakov, V. N.},
   title={The thresholding greedy algorithm, greedy bases, and duality},
   journal={Constr. Approx.},
   volume={19},
   date={2003},
   number={4},
   pages={575--597},
}

\bib{DilworthMitra2001}{article}{
   author={Dilworth, S. J.},
   author={Mitra, D.},
   title={A conditional quasi-greedy basis of $l_1$},
   journal={Studia Math.},
   volume={144},
   date={2001},
   number={1},
   pages={95--100},
}

\bib{GHO2013}{article}{
   author={Garrig{\'o}s, G.},
   author={Hern{\'a}ndez, E.},
   author={Oikhberg, T.},
   title={Lebesgue-type inequalities for quasi-greedy bases},
   journal={Constr. Approx.},
   volume={38},
   date={2013},
   number={3},
   pages={447--470},
}


\bib{GW2014}{article}{
   author={Garrig{\'o}s, G.},
   author={Wojtaszczyk, P.},
   title={Conditional quasi-greedy bases in Hilbert and Banach spaces},
   journal={Indiana Univ. Math. J.},
   volume={63},
   date={2014},
   number={4},
   pages={1017--1036},
}
 \bib{Wo2000}{article}{
 author={Wojtaszczyk, P.},
 title={Greedy algorithm for general biorthogonal systems},
 journal={J. Approx. Theory},
 volume={107},
 date={2000},
 number={2},
 pages={293--314},
}


 
 \end{biblist}
\end{bibsection}
\end{document}